\documentclass[11pt,a4paper,twoside,reqno]{amsart}
\author[L. \'Alvarez-C\'onsul]{Luis \'Alvarez-C\'onsul}
\address{Instituto de Ciencias Matem\'aticas (CSIC-UAM-UC3M-UCM)\\ Nicol\'as Cabrera 13--15, Cantoblanco\\ 28049 Madrid, Spain}
\email{l.alvarez-consul@icmat.es}

\author[M. Garcia-Fernandez]{Mario Garcia-Fernandez}
\address{Universidad Aut\'onoma de Madrid and
	Instituto de Ciencias Matem\'aticas (CSIC-UAM-UC3M-UCM)\\ Ciudad
	Universitaria de Cantoblanco\\ 28049 Madrid, Spain}
\email{mario.garcia@icmat.es}

\author[O. Garc\'{\i}a-Prada]{Oscar Garc\'{\i}a-Prada}
\address{Instituto de Ciencias Matem\'aticas (CSIC-UAM-UC3M-UCM)\\ Nicol\'as Cabrera 13--15, Cantoblanco\\ 28049 Madrid, Spain}
\email{oscar.garcia-prada@icmat.es}

\author[V. P. Pingali]{Vamsi Pritham Pingali}
\address{Department of Mathematics, Indian Institute of Science, Bangalore, India - 560012}
\email{vamsipingali@iisc.ac.in}

\author[C.-J. Yao]{Chengjian Yao}
\address{ Institute of Mathematical Sciences, ShanghaiTech University, 393 Middle Huaxia Road, Pudong, 
	Shanghai, 201210 China.}
\email{yaochj@shanghaitech.edu.cn}

\thanks{
The work of the first three authors was partially supported by the Spanish Ministry of Science and Innovation, through the `Severo Ochoa Programme for Centres of Excellence in R\&D' (CEX2019-000904-S) and under grants PID2019-109339GB-C31, PID2019-109339GA-C32, and EUR2020-112265. The fourth author is partially supported by an SERB MATRICS grant : MTR/2020/000100 and also by grant F.510/25/CAS-II/2018(SAP-I) from UGC (Govt. of India).}

\usepackage{hyperref,url, bm}
\usepackage{amssymb,latexsym}
\usepackage{amsmath,amsthm}
\usepackage[latin1]{inputenc}
\usepackage{enumerate}
\usepackage{color}
\usepackage{venndiagram}
\usepackage[all]{xy} \CompileMatrices
\SelectTips{cm}{12} 

\usepackage{verbatim} 

\usepackage{wrapfig}
\usepackage{graphicx}
\usepackage{color}
\usepackage{tikz}
\usepackage{pgfplots}

\def\YYint#1#2#3{{\setbox0=\hbox{$#1{#2#3}{\int}$}
		\vcenter{\hbox{$#2#3$}}\kern-.52\wd0}}

\theoremstyle{plain}
\newtheorem{theorem}{Theorem}[section]
\newtheorem{lemma}[theorem]{Lemma}

\newtheorem{proposition}[theorem]{Proposition}

\newtheorem*{theorem*}{Theorem}
\theoremstyle{definition}
\newtheorem{definition}[theorem]{Definition}
\newtheorem{definition-theorem}[theorem]{Definition-Theorem}

\theoremstyle{remark}
\newtheorem{remark}[theorem]{Remark}

\numberwithin{equation}{section} 

\newcommand{\dbar}{\bar{\partial}}

\newcommand{\PP}{{\mathbb P}}

\newcommand{\surj}{\to\kern-1.8ex\to}

\newcommand{\cF}{\mathcal{F}}

\title[Self-dual Einstein-Maxwell-Higgs equations on a compact surface]{Obstructions to the existence of solutions of the self-dual Einstein-Maxwell-Higgs equations on a compact surface}
\begin{document}
	\maketitle

\begin{abstract}
In this note we present an obstruction to the existence of solutions to the self-dual Einstein-Maxwell-Higgs equations on a compact surface, which depends on the multiplicities of the zeroes of the \emph{Higgs field} and the \emph{vortex number} $N$. In particular, we exhibit infinitely many examples of Higgs fields for which solutions cannot exist.

\end{abstract} 
\vspace{0.5cm}

\section{Introduction}

Let $S$ be a smooth oriented compact surface, and $\{p_1, \cdots, p_d\}$ be a collection a distinct points on $S$ with positive integer multiplicities $n_1, \cdots, n_d \in \mathbb{N}$, respectively. Give a smooth metric $g$ on $S$ and smooth function $u:S\backslash\{p_1, \cdots, p_d\}\to \mathbb{R}$ defined away from the points consider the system of PDE
\begin{equation}\label{eq:Yangreduction}
\begin{split}
\Delta_g u & = \frac{1}{\varepsilon^2}(e^u - \tau^2 )+4\pi\sum_{j=1}^d n_j \delta_{p_j},\\
K_g  & = -\frac{a}{2 \tau^2}\left[ \frac{\tau^2}{\varepsilon^2}(e^u-\tau^2) -\Delta_g e^u  \right].
\end{split}
\end{equation}
Here $\Delta_g$ is Laplace-Beltrami operator induced by the metric $g$, $K_g$ is the Gaussian curvature of $g$, $\delta_{p_j}$ is the Dirac distribution on $(S,\omega)$ concentrated at $p_j$, where $\omega$ is the area (volume) form of $g$, and $\tau,\varepsilon, a$ are positive real constants. With our convention, the Laplacian has negative eigenvalues. As proved by Y. Yang, the system \eqref{eq:Yangreduction} is equivalent to the self-dual Einstein-Maxwell-Higgs equations arising in theoretical physics \cite{Yang1992,Yang}.

The aim of this note is to present an obstruction to the existence of solutions of \eqref{eq:Yangreduction}, and hence to the self-dual Einstein-Maxwell-Higgs equations on a compact surface, which depends on the configuration of points $p_j$ and their multiplicities $n_j$. Observe that, taking $g = e^\eta g_0$ for a smooth function $\eta$ and fixed background metric $g_0$ and
using the formula for the change in the scalar curvature
$$
2K_g = e^{-\eta}(2K_{g_0} - \Delta_{g_0} \eta),
$$
the previous system is equivalent to (cf. \cite[Equation (1.1)]{HLS})
\begin{equation}\label{eq:Yangreduction2}
\begin{split}
\Delta_{g_0} u & = \frac{1}{\varepsilon^2}e^\eta(e^u - \tau^2 )+4\pi\sum_{j=1}^d n_j \delta_{p_j},\\
\Delta_{g_0}(\eta + \tfrac{a}{\tau^2} e^u)  & = 2K_{g_0} + \frac{a}{\varepsilon^2}e^\eta (e^u - \tau^2).
\end{split}
\end{equation}
Here $\delta_{p_j}$ is the Dirac distribution on $(S,\omega_0)$ concentrated at $p_j$, where $\omega_0$ is the area form of $g_0$. Define the \emph{vortex number} by
$$
N = \sum_{j=1}^d n_j.
$$
Or main result can be stated as follows. 

\begin{theorem}\label{thm:dleq2}

Suppose that \eqref{eq:Yangreduction2} admits a solution $(u,\eta)$ with $d \leq 2$ for some parameter $\varepsilon>0$ and some smooth background metric $g_0$ on $S$, where $\eta$ is a smooth function on $S$ and $u:S\backslash\{p_1, \cdots, p_d\}\to \mathbb{R}$ is a smooth function defined away from the points. Then, there must hold that
$$
d = 2 \quad \textrm{and}  \quad n_1 = n_2.
$$
\end{theorem}

Our main theorem gives infinitely many examples of configurations of points for which there cannot be solutions to the self-dual Einstein-Maxwell-Higgs equations on a compact surface, in apparent contradiction with \cite[Theorem 1.2]{HLS} (see Remark \ref{rem:HLS} and the Erratum \cite{HLSErratum}). Our proof exploits systematically the complex geometry of the \emph{Higgs field} $\phi$ in the self-dual Einstein-Maxwell-Higgs equations, which is implicit in the system of equations \eqref{eq:Yangreductiond=2} (see Section \ref{sec:pre}).

Theorem \ref{thm:dleq2} is a particular case of a more general result in previous work by the first four authors of this note, published online in \emph{Mathematische Annalen} in February 2020 (see \cite[Theorem 1.3 and Proposition 2.6]{AlGaGaPi}) and previously announced in \cite{AlGaGa2}. In the present setup, \cite[Theorem 1.3]{AlGaGaPi} can be stated as follows.

\begin{theorem}[\cite{AlGaGaPi}]\label{th:AlGaGaPi}

Suppose that \eqref{eq:Yangreduction2} admits a solution $(u,\eta)$ for some parameter $\varepsilon>0$ and some smooth background metric $g_0$ on $S$, where $\eta$ is a smooth function on $S$ and $u:S\backslash\{p_1, \cdots, p_d\}\to \mathbb{R}$ is a smooth function defined away from the points. Then, one of the following holds:
\begin{enumerate}
\item[\textup{(1)}] $d > 2$ and  $n_j < \frac{N}{2}$ for all $j$,
\item[\textup{(2)}] $d = 2$ and $2n_1 = 2n_2 = N$.
\end{enumerate}
\end{theorem}

The previous result is a converse to \emph{Yang's Existence Theorem} (see \cite[Theorem 1.2]{Yang} and \cite[Theorem 1.1]{Yang3}). Combined with \cite[Theorem 1.1]{FPY}, which carefully studies the possible volumes of the normalized metrics $\tfrac{1}{\varepsilon^2}g$ solving \eqref{eq:Yangreduction}, it gives a complete solution to the existence problem for the self-dual Einstein-Maxwell-Higgs equations on a compact surface.

Observe that Theorem \ref{thm:dleq2} follows directly by application of Theorem \ref{th:AlGaGaPi}, from the assumption $d \leq 2$. As a matter of fact, our method of proof here is, as in \cite{AlGaGaPi}, the evaluation of a \emph{Futaki invariant} for the self-dual Einstein-Maxwell-Higgs equations. This invariant was first constructed in \cite{AlGaGaPi} via the general theory for the K\"ahler-Yang-Mills equations developed in \cite{AlGaGa}, and follows the same principles as the classical Futaki invariant in the K\"ahler-Einstein theory \cite{Futaki1}. The main novelty of the present note is an alternative, independent construction of the Futaki invariant, by direct and explicit calculations. We have  deliberatively avoided the use of the general moment map theory in \cite{AlGaGa,AlGaGaPi} as well as any technicalities coming from the geometry of the self-dual Einstein-Maxwell-Higgs equations, trying to make the exposition self-contained and more accessible to the PDE community.

\section{An obstruction to the self-dual Einstein-Maxwell-Higgs equations}

\subsection{Preliminaries}\label{sec:pre}

Let $S$ be a compact connected smooth surface. We fix $\{p_1, \cdots, p_d\}$ a collection of distinct points on $S$ with positive integer multiplicities $n_1, \cdots, n_d \in \mathbb{N}$, respectively. We also fix positive real constants $\tau,\varepsilon, a > 0$. By Gauss-Bonnet Theorem, the existence of solutions of \eqref{eq:Yangreduction} implies that
$$
\chi(S) = a N > 0.
$$
Therefore, $S$ must be diffeomorphic to the two-sphere $S^2$ and furthermore one has the following `quantization condition'
$$
a = \frac{2}{N}.
$$
We fix a background metric $g_0$ on $S^2$ with volume $V$. One of the main methods introduced in \cite{AlGaGaPi} for the proof of Theorem \ref{th:AlGaGaPi} is the construction of an invariant, called the \emph{Futaki invariant}, which only depends on the conformal class of $g_0$, the volume $V$, and the tuple $(\{p_j\},\{n_j\},\tau,a)$, which vanishes identically provided that there exists a solution of \eqref{eq:Yangreduction2}. To recall the definition of the Futaki invariant, 
in this note we will focus on the special case 
$$
d \leq 2.
$$
Without loss of generality, we assume that $p_1$ and $p_2$ are the north and south pole on $S^2$, respectively, and take $g_0$ to be the round metric on $S^2$ with total volume $V$. Given a positive integer $0 \leq \ell \leq N$, we consider the following special case of equation \eqref{eq:Yangreduction2}
\begin{equation}\label{eq:Yangreductiond=2}
\begin{split}
\Delta_{g_0} u & = \frac{1}{\varepsilon^2}e^\eta(e^u - \tau^2 )+4\pi(N-\ell) \delta_{p_1} + 4\pi \ell \delta_{p_2},\\
\Delta_{g_0}(\eta + \tfrac{a}{\tau^2} e^u)  & = 2K_{g_0} + \frac{a}{\varepsilon^2}e^\eta (e^u - \tau^2).
\end{split}
\end{equation}
The construction of the Futaki invariant exploits systematically the complex geometry of the \emph{Higgs field} $\phi$ in the self-dual Einstein-Maxwell-Higgs equations, which is implicit in the system of equations \eqref{eq:Yangreductiond=2}. In order to use this geometry, in the sequel we identity the two-sphere with the complex Riemann sphere
$$
S^2 \cong \mathbb{P}^1.
$$
Taking the stereographic projection from the north pole, we have
\begin{equation}
\label{def:g0}
g_0 = \frac{V}{\pi}\frac{dx^2 + dy^2}{(1 + x^2 + y^2)^2}.
\end{equation} 
In the complex coordinate $z = x + i y$, the Higgs field of our interest is
\begin{equation}\label{eq:Higgs}
\phi = z^\ell.
\end{equation}
To define the Futaki invariant, we first write \eqref{eq:Yangreductiond=2} in a different form. Define a smooth function $\Phi$ by
$$
\Phi = \frac{|z|^{2\ell}}{(1+|z|^2)^N}.
$$
By a simple change of holomorphic coordinates $z \to z^{-1}$, it is easy to see that $\Phi$ extends to a smooth function on $\mathbb{P}^1$. More invariantly, $\phi$ defines a global holomorphic section of the line bundle $L = \mathcal{O}_{\mathbb{P}^1}(N)$ over $\mathbb{P}^1$, which vanishes at the points $p_1 \equiv \infty$ and $p_2 \equiv 0$ with multiplicities $N-\ell$ and $\ell$, respectively. In the holomorphic coordinate $z$ the Fubini-Study metric on $L$ reads
$$
h_{FS}^N = \frac{1}{(1+|z|^2)^N}
$$
and therefore $\Phi = |\phi|^2_{h_{FS}^N}$. For the next formula, notice that with our convention the Laplacian has negative eigenvalues. 

\begin{lemma}\label{lem:PoincareLelong}
$$
\Delta_{g_0} \log \Phi = - \frac{4\pi N}{V} + 4\pi(N-\ell) \delta_{p_1} + 4\pi \ell  \delta_{p_2}. 
$$ 
\end{lemma}

\begin{proof}
The claim follows now from the Poincar\'e-Lelong formula
\begin{align*}
i \partial \dbar \log |\phi|^2_{h_{FS}^N} & = - i F_{h_{FS}^N} + 2\pi(N-\ell) \delta_{p_1} + 2\pi \ell  \delta_{p_2}\\
& = - N \frac{i dz \wedge d \overline{z}}{(1+|z|^2)^2} + 2\pi(N-\ell) \delta_{p_1} + 2\pi \ell\delta_{p_2}.
\end{align*}
\end{proof}

Using the unknown function $u$ in \eqref{eq:Yangreductiond=2}, we define a smooth function on the sphere by
$$
f = u - \log \Phi.
$$
Using the previous Lemma, the system \eqref{eq:Yangreductiond=2} now reads
\begin{equation}\label{eq:Yangreductiond=2BIS}
\begin{split}
\Delta_{g_0} f & = \frac{1}{\varepsilon^2}e^\eta(e^f \Phi - \tau^2 ) + \frac{4\pi N}{V},\\
\Delta_{g_0}(\eta + \tfrac{a}{\tau^2} e^f \Phi)  & = 2K_{g_0} + \frac{a}{\varepsilon^2}e^\eta (e^f \Phi - \tau^2).
\end{split}
\end{equation}
Consider the radial vector field
$$
\bm v = \bm v^{1,0} +\bm v^{0,1} = z \partial_z + \overline{z} \partial_{\overline{z}}.
$$
Of course, this is the local expression of a globally defined real holomorphic vector field on $\PP^1$ whose flow evolves along the meridians of $S^2 \cong \PP^1$, and vanishes at the poles. Let $\eta$ be a smooth function on $\PP^1$. Using that $\bm v$ is holomorphic we have
$$
\dbar (e^\eta \iota_{\bm v^{1,0}} \omega_0) = 0.
$$
Since $\PP^1$ is simply connected, the Hodge decomposition of the de Rham cohomology ensures the existence of a global $\dbar$-potential for $e^\eta \iota_{\bm v^{1,0}}\omega_0$.

\begin{definition}
Let $\eta$ be a smooth real function on $\PP^1$. We define $\varphi_\eta$ as the unique smooth complex-valued function on $\mathbb{P}^1$ satisfying
$$
\dbar \varphi_\eta = e^\eta \iota_{\bm v^{1,0}}\omega_0, \qquad \int_{\mathbb{P}^1} \varphi_\eta e^\eta \omega_0 = 0.
$$ 
\end{definition}

We next construct a complex function $\psi_f$ associated to the other unknown of the equation \eqref{eq:Yangreductiond=2BIS}. Let $\omega_{FS}$ denote the Fubini-Study K\"ahler form
$$
\omega_{FS} := i \dbar \partial \log h_{FS}^1 = \frac{i dz \wedge d \overline{z}}{(1+|z|^2)^2}.
$$
Given a smooth function $f$ on $\PP^1$, similarly as before we have
$$
\dbar(\iota_{\bm v^{1,0}} (N\omega_{FS} + i \dbar \partial f)) = 0 
$$
and hence $\iota_{\bm v^{1,0}} (N\omega_{FS} + i \dbar \partial f)$ admits a smooth global $\dbar$-potential. The choice of normalization for this potential is a delicate point, which we consider next. Geometrically, this is choosen so that $(\bm v^{1,0},\psi_f)$ defines an \emph{infinitesimal automorphism} of the Higgs field $\phi = z^\ell$, in the following sense. 

\begin{lemma}\label{eq:Psif}
Let $f$ be a smooth real function on $\PP^1$. We define $\psi_f$ as the smooth complex-valued function on $\mathbb{P}^1$ given by
$$
\psi_f = \ell - N \frac{|z|^2}{1 + |z|^2} + \partial f(\bm v^{1,0}).
$$
Then, if we set $h = e^f h_{FS}^N$, the following identities hold
\begin{equation}\label{eq:potentialfstructure}
\begin{split}
- i \dbar \psi_f &= \iota_{\bm v^{1,0}} (i \dbar \partial \log h),\\
\iota_{\bm v^{1,0}} \left( \partial \phi + (\partial \log h)\phi \right) & = \psi_f \phi.
\end{split}
\end{equation}
\end{lemma}

\begin{proof}
Observe that $\psi_0 = \ell - N \frac{|z|^2}{1 + |z|^2}$ defines a smooth global function on $\PP^1$, which satisfies
$$
\dbar \psi_0 = - N \frac{z d \overline{z}}{(1+|z|^2)^2} = iN \iota_{\bm v^{1,0}}\omega_{FS}.
$$
From this, using that $\bm v$ is holomorphic we obtain
$$
\dbar \psi_f = iN \iota_{\bm v^{1,0}}\omega_{FS} - \iota_{\bm v^{1,0}} \dbar \partial f.
$$
Similarly
$$
\iota_{\bm v^{1,0}} \left( \partial \phi + (\partial \log e^f h_{FS}^N)\phi \right) = \iota_{\bm v^{1,0}}\Bigg{(}\ell z^{\ell - 1} dz- N\frac{\overline{z}dz}{1+ |z|^2} \phi + (\partial f)\phi \Bigg{)} = \psi_f \phi.
$$
\end{proof}

We will also need the following technical formula.

\begin{lemma}\label{lem:dbarPsifmixed}
Set $h = e^f h_{FS}^N$. Then the following identity holds
$$
\dbar \left( \psi_f(e^f\Phi - \tau^2) \right) =  - \iota_{\bm v^{1,0}} \left(\dbar \partial (e^f\Phi) - \tau^2 \dbar \partial \log h \right).
$$
\end{lemma}
\begin{proof}
Applying the structural formulae for $\psi_f$ in Lemma \ref{eq:Psif}, we obtain
\begin{align*}
\dbar \left( \psi_f(e^f\Phi - \tau^2) \right) &= \dbar \left( \psi_f\phi h \overline{\phi} - \tau^2 \psi_f \right)\\
& = \dbar \left( \iota_{\bm v^{1,0}}\left( \partial \phi + (\partial \log h)\phi \right)h \overline{\phi} \right) + \tau^2  \iota_{\bm v^{1,0}} (\dbar \partial \log h)\\
& = - \iota_{\bm v^{1,0}} \left( \dbar \partial (e^f\Phi) - \tau^2 \dbar \partial \log h \right).
\end{align*}
\end{proof}

\subsection{A Futaki invariant for the self-dual Einstein-Maxwell-Higgs equations}

In this section we introduce the definition of the Futaki invariant for the self-dual Einstein-Maxwell-Higgs equations, following \cite{AlGaGaPi} (see Remark \ref{rmk:Futaki}). As a warm-up, we recall first the definition of the classical Futaki invariant in our setup. This classical invariant, originally introduced by A. Futaki \cite{Futaki1}, provides an obstruction to the existence of K\"ahler-Einstein metrics on a compact complex manifold. Since we are dealing with the Riemann sphere, we will simply prove that the classical Futaki invariant identically vanishes.

\begin{lemma}[Classical Futaki invariant]\label{lem:Futakiclas}
The following expression does not depend on the choice of smooth function $\eta$:
\begin{align*}
\cF_{0} = -   \int_{S^2} \varphi_\eta \Big{(}2K_{g_0} - \Delta_{g_0} \eta \Big{)} \omega_0.
\end{align*}
Consequently, it vanishes identically.
\end{lemma}

\begin{proof}
We claim that $\cF_0 = \cF_0(\eta)$ is independent of $\eta$, provided that $\int_{S^2} e^\eta \omega_0 = V$. If this is the case, given such $\tilde \eta$ and using that $g_0$ has constant curvature
$$
\cF_{0}(\tilde \eta) = -  2K_{g_0} \int_{S^2} \varphi_0  \omega_0 = 0
$$ 
by definition of $\varphi_0$. Furthermore,
$$
\varphi_{\eta + \log t} = t \varphi_{\eta}
$$
for any $\eta$, which implies $\cF_0(\tilde \eta + \log t) = t \cF_0(\tilde \eta) = 0$.

To prove our initial claim, we set $g := e^\eta g_0$ and $\varphi_g : = \varphi_\eta$, and note that
$$
 \cF_0(g):= \cF_{0}(\eta) = - 2\int_{S^2} \varphi_g \Big{(}K_{g} - \hat K \Big{)} \omega = - 2\int_{S^2} \varphi_g \Big{(}\rho_{g} - \hat K \omega \Big{)}
$$
where $\hat K = 4\pi/V$, $\omega$ denotes the volume form corresponding to $g$, and $\rho_{g} = -i\partial\bar\partial \log \det g$ denotes the Ricci form. Using that $\int_{S^2} \omega = V$ we can write 
$$
\omega = e^\eta \omega_0 = (1 + \Delta_{g_0} u) \omega_0
$$
for some smooth function $u$. Since the space of metrics conformal to $g_0$ and with fixed volume is path-connected, it suffices to prove that the variation of $\cF_0 = \cF_0(u)$ with respect to $u$ vanishes identically. Taking a path $\omega_t = \omega +  2i \partial \dbar u_t$ and denoting $\frac{d}{dt}u_t=\dot u$, we calculate
$$
\dot \varphi_g = 2i \bm v^{1,0}(\dot u), \qquad \dot \rho_{g} = i \dbar \partial \Delta_g \dot u.
$$
The second identity is straightforward. As for the first identity, from the defining equation (we set $\varphi = \varphi_g$ for simplicity)
\[
\varphi_{\bar\lambda} = i g_{\mu\bar\lambda}v^\mu, 
\]
taking the derivative with respect to $t$, we get 
\[
\dot\varphi_{\bar\lambda}
=
2i \dot u_{\mu\bar\lambda} v^\mu
=
2i (v^\mu \dot u_\mu)_{\bar\lambda}
\]
by the holomorphicity of the vector field $\bm v^{1,0}=v^\mu\partial_\mu$, thus $\dot \varphi= 2i \bm v^{1,0}(\dot u)+c$ for some constant $c$. 	
On the other hand, the normalization condition $\int_{S^2} \varphi \omega = 0$ implies 
\[
0
=
\int_{S^2} \dot \varphi \omega 
+ \varphi \cdot 2i\partial\bar\partial \dot u
=
\int_{S^2} \dot\varphi \omega - 2i \partial \dot u \wedge \bar\partial \varphi
= 
\int_{S^2} 2i v^\mu \dot u_\mu \cdot i g_{z\bar z}dz\wedge d\bar z + c\omega 
- 
2i \int_{S^2} \dot u_z i g_{z\bar z}v^z dz\wedge d\bar z, 
\]
therefore $c=0$ and $\dot \varphi_g = 2i v^\mu \dot u_\mu=2i \bm v^{1,0}(\dot u)$. 

We will also need the following identity 
$$
\dbar \Delta_g \varphi_g = - 2\iota_{\bm v^{1,0}} \rho_{g}.
$$ To prove this, we calculate

\begin{align*}
	\bar\partial \Delta_g \varphi_g 
	& = 
	2\bar\partial (g^{\alpha\bar\beta}\varphi_{\alpha\bar\beta})\\
	& = 
	2 g^{\alpha\bar\beta}(\partial_{\bar\kappa} \varphi_{\alpha\bar\beta}-g^{\alpha\bar\mu}g^{\lambda\bar\beta}\partial_{\bar \kappa} g_{\lambda\bar\mu}\varphi_{\alpha\bar\beta}) d\bar z^\kappa\\
	& = 
	2 g^{\alpha\bar\beta}\nabla_{\bar\kappa}\nabla_{\bar\beta}\nabla_\alpha\varphi \cdot d\bar z^\kappa\\
	& = 
	2g^{\alpha\bar\beta}
	\left(
	-R_{\bar\kappa \alpha\bar \beta}^{\phantom{\bar\kappa\alpha\bar\beta}\bar\lambda}\nabla_{\bar\lambda}\varphi 
	+ 
	\nabla_\alpha \nabla_{\bar\kappa}\nabla_{\bar\beta}\varphi
	\right)d\bar z^\kappa\\
	& = 
	-2R_{\mu\bar\kappa} g^{\mu\bar\lambda}\nabla_{\bar\lambda}\varphi d\bar z^\kappa,
	\end{align*}
in which we use the relation $\nabla_{\bar \kappa}\nabla_{\bar\beta} \varphi= \nabla_{\bar\kappa} \left( i g_{\alpha\bar\beta}v^\alpha\right)=ig_{\alpha\bar\beta} \nabla_{\bar\kappa}v^\alpha=0$ by the holomorphicity of $\bm v^{1,0}$. 	Finally, using again the formula $\bar\partial \varphi = \iota_{\bm v^{1,0}} \omega$, i.e. $\varphi_{\bar\lambda} = i g_{\mu\bar\lambda}v^\mu$, the above equation yields 
\[
\bar\partial \Delta_g \varphi_g 
= 
-2 i R_{\mu\bar\kappa} v^\mu d\bar z^\kappa
= 
-2\iota_{\bm v^{1,0}} \rho_g.
\]
Using the various identities above, we conclude with the following calculation
\begin{align*}
- \frac{d}{dt} \cF_{0}(g_t) & = 2i \int_{S^2} \bm v^{1,0}(\dot u) \Big{(}\rho_{g} - \hat K \omega \Big{)} 
+
\int_{S^2} \varphi_g 
\left( - i \partial\bar\partial \Delta_g \dot u - \hat K\cdot 2i\partial\bar\partial\dot u\right)\\
&= 
2i \int_{S^2} v^z \dot u_z (\rho_g-\hat K \omega)+ 
\int_{S^2} \varphi_g 
\left( - i \partial\bar\partial \Delta_g \dot u - \hat K\cdot 2i\partial\bar\partial\dot u\right)\\
& = 
2i \int_{S^2} \dot u_z dz \wedge \iota_{\bm v^{1,0}} (\rho_g - \hat K \omega) + 
\int_{S^2} \varphi_g 
\left( - i \partial\bar\partial \Delta_g \dot u - \hat K\cdot 2i\partial\bar\partial\dot u\right)\\
& = 
i \int_{S^2} \dot u_z dz \wedge (-\bar\partial \Delta_g \varphi_g - 2\hat K \bar\partial \varphi_g) 
+
\int_{S^2} \varphi_g 
\left( - i \partial\bar\partial \Delta_g \dot u - \hat K\cdot 2i\partial\bar\partial\dot u\right)\\
& = 
\frac{1}{2}\int_{S^2} \Delta_g \dot u \left( \Delta_g \varphi_g + 2\hat K \varphi_g\right) \omega
-
\frac{1}{2} \int_{S^2} \left( \Delta_g \varphi_g \Delta_g \dot u + 2\hat K \varphi_g\Delta_g \dot u  \right)\omega\\
& = 0.
\end{align*}

\end{proof}

We are ready to recall the definition of the Futaki invariant for the self-dual Einstein-Maxwell-Higgs equations, originally introduced in \cite{AlGaGaPi}.

\begin{proposition}[Einstein-Maxwell-Higgs Futaki invariant]\label{prop:Futaki}
The following expression does not depend on the choice of smooth functions $\eta$ and $f$
\begin{align*}
\cF_{\ell,N,V,\tau} & = 2i\frac{a}{\tau^2}\int_{S^2} \psi_f \left( \frac{4\pi N}{V} - \Delta_{g_0} f + e^\eta(e^f \Phi - \tau^2) \right) \omega_0\\
& -   \int_{S^2} \varphi_\eta \left( 2K_{g_0} - \Delta_{g_0}(\eta + \tfrac{a}{\tau^2} e^f \Phi) - a \left(  \frac{4\pi N}{V} - \Delta_{g_0} f\right) \right) \omega_0,
\end{align*}
provided that $\int_{S^2} e^\eta \omega_0 = V$. Consequently, it defines an invariant which vanishes if the equations \eqref{eq:Yangreductiond=2BIS} admit a smooth solution $(\eta,f)$ with $\int_{S^2} e^\eta \omega_0 = V$ and $\varepsilon^2 = 1$.
\end{proposition}

\begin{proof}
We denote $\cF(\eta,f) = \cF_{\ell,N,V,\tau}$, as above. By Lemma \ref{lem:Futakiclas}, we have a simplified expression
\begin{align*}
\cF(\eta,f) & = 2i\frac{a}{\tau^2}\int_{S^2} \psi_f \left(  \frac{4\pi N}{V} - \Delta_{g_0} f + e^\eta(e^f \Phi - \tau^2) \right) \omega_0\\
&\qquad + a  \int_{S^2} \varphi_\eta \left(  \tfrac{1}{\tau^2}\Delta_{g_0}( e^f \Phi) + \frac{4\pi N}{V} - \Delta_{g_0} f \right) \omega_0.
\end{align*}
Using that $\int_{S^2} e^\eta \omega_0 = V$, we have $e^\eta \omega_0 = (1 + \Delta_{g_0}u)\omega_0$ for some smooth function $u$. As in the proof of Lemma \ref{lem:Futakiclas}, denoting $\omega = (1 + \Delta_{g_0}u)\omega_0$ we can write the previous expression as
\begin{align*}
\cF(u,f) & = 4i\frac{a}{\tau^2}\int_{S^2} \psi_f \left( (N \omega_{FS}  + i \dbar \partial f) + \frac{1}{2}(e^f \Phi - \tau^2)\omega \right)\\
& \qquad + a  \int_{S^2} \varphi_g \left(  \tfrac{1}{\tau^2}\Delta_{g}( e^f \Phi)\omega + 2 (N \omega_{FS}  + i \dbar \partial f)  \right).
\end{align*}
Taking variations with respect to $u$ and applying Lemma \ref{lem:dbarPsifmixed} we have
\begin{align*}
\delta_u \cF 
& = 
2i\frac{a}{\tau^2}\int_{S^2} \psi_f \Big{(} (e^f \Phi - \tau^2)\cdot 2i\partial \dbar \dot u \Big{)} 
+ 
2i\frac{a}{\tau^2}  \int_{S^2} \bm v^{1,0}(\dot u ) \left(  2i \partial \dbar( e^f \Phi) + 2 \tau^2 (N \omega_{FS}  + i \dbar \partial f)\right)\\
& = 
2i\frac{a}{\tau^2}\int_{S^2} \dot u \cdot 2i \partial \dbar \left( \psi_f  (e^f \Phi - \tau^2) \right) 
+ 
2i\frac{a}{\tau^2}  \int_{S^2} \partial \dot u  \wedge \iota_{\bm v^{1,0}} \left(  2i \partial \dbar( e^f \Phi) + 2 \tau^2 (N \omega_{FS}  + i \dbar \partial f)\right)\\
& =
 2i\frac{a}{\tau^2}\int_{S^2} \dot u \cdot 2i \partial \Bigg(  \dbar \left( \psi_f  (e^f \Phi - \tau^2) \right) - \iota_{\bm v^{1,0}}(\partial \dbar( e^f \Phi) + \tau^2 \dbar \partial \log h)   \Bigg)\\
& = 0,
\end{align*}
where we have used that $i \dbar \partial \log h = N \omega_{FS}  + i \dbar \partial f$. Similarly, taking variations with respect to $f$, gives
\begin{align*}
\delta_f \cF & = 4i\frac{a}{\tau^2}\int_{S^2} \partial \dot f (\bm v^{1,0}) \left( i \dbar \partial \log h  + \frac{1}{2}(e^f \Phi - \tau^2)\omega \right)
 + 4i\frac{a}{\tau^2}\int_{S^2} \psi_f \left(  i \dbar \partial \dot f + \frac{1}{2} \dot f e^f \Phi\omega \right)\\
&\qquad
 + \frac{a}{\tau^2}  \int_{S^2} \varphi_g \cdot 2i \partial \dbar \left( \dot f (e^f \Phi - \tau^2) \right)\\
& = - 4i\frac{a}{\tau^2}\int_{S^2} \dot f \cdot \partial \left(  \iota_{\bm v^{1,0}} \left( i\dbar \partial \log h  + \frac{1}{2}(e^f \Phi - \tau^2)\omega \right) \right)
+ 4i\frac{a}{\tau^2}\int_{S^2} \left( \dot f \cdot i \dbar \partial \psi_f + \frac{1}{2}\psi_f \dot f e^f \Phi \omega \right)\\
&\qquad
 +2i \frac{a}{\tau^2}  \int_{S^2} \dot f (e^f \Phi - \tau^2)\cdot \partial (\iota_{\bm v^{1,0}} \omega)\\
& =
 - 4i\frac{a}{\tau^2}
 \int_{S^2} \dot f \cdot \partial \left( \iota_{\bm v^{1,0}} ( i\dbar \partial \log h) \right) 
+ 
\frac{1}{2}\dot   f\cdot \partial (e^f \Phi) \wedge  \iota_{\bm v^{1,0}} \omega
+ \frac{1}{2}\dot f (e^f\Phi-\tau^2) \partial ( \iota_{\bm v^{1,0}}\omega)  \\
&\qquad 
 + 4i\frac{a}{\tau^2}\int_{S^2} \dot f \cdot \partial \left( \iota_{\bm v^{1,0}} (i \dbar \partial \log h) \right)
 + \frac{1}{2} \dot f \cdot \psi_f e^f\Phi \omega 
  + \frac{1}{2}\dot f (e^f\Phi-\tau^2) \partial\left(\iota_{\bm v^{1,0}}\omega\right)\\
 & = 2i\frac{a}{\tau^2}\int_{S^2} \dot f 
\left( \psi_f e^f\Phi \omega 
- 
\partial (e^f\Phi) \wedge \iota_{\bm v^{1,0}} \omega
\right)\\
& = 
2i\frac{a}{\tau^2} \int_{S^2} 
\dot f \Big(  
\psi_f \phi h\overline\phi - \iota_{\bm v^{1,0}} \left( \partial\phi + \phi\partial \log h\right)h \overline\phi 
\Big) \omega\\
& = 0, 
\end{align*}
where in the third inequality we used the first item of Equation \eqref{eq:potentialfstructure} and in the last inequality we used the definition $e^f\Phi=\phi h\overline\phi$ and the second item of Equation \eqref{eq:potentialfstructure}.

\end{proof}

\begin{remark}\label{rmk:Futaki}
More precisely, the Futaki invariant for the self-dual Einstein-Maxwell-Higgs equations constructed in \cite{AlGaGaPi} is a character of the Lie algebra of infinitesimal automorphisms of the triple $(\PP^1,\mathcal{O}_{\mathbb{P}^1}(N),\phi)$. The quantity $\cF_{\ell,N,V,\tau}$ considered here is the evaluation of this character at the infinitesimal automorphism $(\bm v^{1,0},\psi_f)$ constructed in Lemma \ref{eq:Psif}.
\end{remark}

\subsection{Proof of Theorem \ref{thm:dleq2}}

We start calculating an explicit formula for the Futaki invariant.

\begin{lemma}\label{lem:EMHFutakiev}
	The Einstein-Maxwell-Higgs Futaki invariant in Proposition \ref{prop:Futaki} is 
	\[
	\cF_{\ell,N,V,\tau}
	=
	ia \left(  V- \frac{4\pi N}{\tau^2} \right) (N- 2\ell).
	\]
\end{lemma}

\begin{proof}
	Take $f=\eta=0$, and $\omega_0=\frac{V}{2\pi}\omega_{FS}$ in Proposition \ref{prop:Futaki}, we get that 
	\begin{align*}
	\cF_{\ell,N,V,\tau} & =
	2i\frac{a}{\tau^2}\int_{S^2} \psi_0 \left(  2N \omega_{FS} 
	+ \frac{V}{2\pi} (\Phi - \tau^2) \omega_{FS}\right)\\
	&\qquad + \frac{a}{\tau^2}  \int_{S^2} \varphi_0 \Delta_{\omega_{FS}} \Phi \omega_{FS}\\
	& =
	2i\frac{a}{\tau^2}\int_{S^2} \psi_0 \left( 2N - \frac{V}{2\pi}\tau^2 \right) \omega_{FS} 
	+ 2i \frac{a}{\tau^2}\cdot \frac{V}{2\pi} 
	\int_{S^2} \Phi \psi_0 \omega_{FS}\\
	&\qquad -  \frac{4a}{\tau^2}  \int_{S^2} \varphi_0 \Phi \omega_{FS}\\
	& = 
	2i \frac{a}{\tau^2} \left( 2N- \frac{V}{2\pi}\tau^2 \right) \int_{S^2}\psi_0\omega_{FS} 
	+ 
	2i \frac{a}{\tau^2} \frac{V}{2\pi} \int_{S^2} 
	\Phi \left( \psi_0 - \frac{|z|^2-1}{|z|^2+1} \right) \omega_{FS}.
	\end{align*}
	where $\psi_0 = \ell - N\frac{|z|^2}{1+|z|^2}$, and $\varphi_0 = \frac{V}{2\pi}\cdot \frac{i}{2}\frac{|z|^2-1}{|z|^2+1}$, and we used the fact that $\Delta_{\omega_{FS}} \varphi_0 = -4\varphi_0$. The final formula for the Futaki invariant follows from the calculations 
	\begin{align*}
	\int_{S^2} \psi_0 \omega_{FS}
	& = 
	\int_{\mathbb{C}} \left( \ell - N \frac{|z|^2}{|z|^2+1}\right) \frac{idz\wedge d\bar z}{(|z|^2+1)^2}\\
	& = 
	2\pi  \int_0^\infty \left( (\ell - N)\frac{1}{(1+s)^2}  + \frac{N}{(1+s)^3} \right) ds\\
	& = 
	\pi (2\ell - N), 
	\end{align*}
	and 
	\begin{align*}
	\int_{S^2} \Phi\left(\psi_0 - \frac{|z|^2-1}{|z|^2+1}\right) \omega_{FS} 
	& =
	\int_\mathbb{C} \left( \ell - N \frac{|z|^2}{|z|^2+1} + \frac{1-|z|^2}{1+|z|^2} \right) \frac{|z|^{2\ell}}{(1+|z|^2)^{N+2}} idz\wedge d\bar z \\
	& = 
	4\pi  \int_0^\infty \left( \ell - N \frac{r^2}{r^2+1} + \frac{1-r^2}{1+r^2} \right) \frac{r^{2\ell +1}}{(1+r^2)^{N+2}}dr\\
	& = 
	2\pi \int_0^\infty \frac{(\ell +1)s^\ell + (\ell - N-1)s^{\ell +1}}{(1+s)^{N+3}} d s\\
	& = 
	2\pi \frac{ s^{\ell +1}}{(1+s)^{N+2}} |_0^\infty \\
	& = 0. 
	\end{align*}
\end{proof}

We also need the following necessary condition for the existence of solutions to the self-dual Einstein-Maxwell-Higgs equations. This inequality was originally observed by Noguchi, Bradlow, and Garc\'ia-Prada in the context of Abelian vortices on a compact surface \cite{Brad,G1,G3,Noguchi}. The proof follows directly by integrating the first equation in \eqref{eq:Yangreductiond=2BIS} using the volume form $e^\eta \omega_0$.

\begin{lemma}\label{lem:BNGP}
Assume that there exists a solution $(f,\eta)$ to \eqref{eq:Yangreductiond=2BIS} with $\int_{S^2} e^\eta \omega_0 = V$ and $\varepsilon^2=1$. Then, there holds the following inequality
$$
V > \frac{4\pi N}{\tau^2}.
$$
\end{lemma}

We are ready for the proof of our main result.

\begin{proof}[Proof of Theorem \ref{thm:dleq2}]

Assume $d \leq 2$ and that there exists a solution $(u,\eta)$ of \eqref{eq:Yangreduction2} for some smooth background metric $g_0'$  on $S^2$ and some parameter $\varepsilon>0$. Notice that in the first equation in \eqref{eq:Yangreduction2} we use the volume form $\omega_0'$ of $g_0'$ to define the distributional equation. That is, for any smooth function $\psi$ on $S^2$ one has 
\[
\int_{S^2} \psi \Delta_{g_0'} u \omega_0'  = \int_{S^2} \psi \frac{1}{\varepsilon^2}e^\eta(e^u - \tau^2 ) \omega_0'+ 4\pi\sum_{j=1}^d n_j \psi(p_j).
	\]
Setting $g = \frac{1}{\varepsilon^2}e^\eta g_0'$. and taking now $\omega$ the volume form of the metric $g$, one obtains
\[
\int_{S^2} \psi \Delta_{g} u \omega  = \int_{S^2} \psi (e^u - \tau^2 ) \omega + 4\pi\sum_{j=1}^d n_j \psi(p_j).
	\]
This means that $(u,g)$ is a solution of  
\begin{equation}\label{eq:Yangreductionfinal}
\begin{split}
\Delta_g u & = (e^u - \tau^2 )+4\pi\sum_{j=1}^d n_j \delta_{p_j},\\
K_g  & = -\frac{a}{2 \tau^2}\left[\tau^2(e^u-\tau^2) -\Delta_g e^u  \right],
\end{split}
\end{equation}
where $\delta_{p_j}$ is the Dirac distribution on $(S^2,\omega)$ concentrated at $p_j$ and $\omega$ is the volume form of $g$. Let $V= \int_{S^2} \omega$, and $g_0$ be the constant scalar curvature metric on $S^2$ defined by the formula \eqref{def:g0}. Let $\omega_0$ be the area form of the metric $g_0$. Notice that $\int_{S^2} \omega_0 =V$. Set $f := u - \log \Phi$, where $\Phi$ is as in Lemma \ref{lem:PoincareLelong}. Then, for $\eta'$ the smooth function on $S^2$ such that $\omega = e^{\eta'}\omega_0$, the pair $(f,\eta')$ is a solution of  
\begin{equation*}
\begin{split}
\Delta_{g_0} f & = e^{\eta'}(e^f \Phi - \tau^2 ) + \frac{4\pi N}{V},\\
\Delta_{g_0}(\eta' + \tfrac{a}{\tau^2} e^f \Phi)  & = 2K_{g_0} + a e^{\eta'} (e^f \Phi - \tau^2),
\end{split}
\end{equation*}
such that $\int_{S^2} e^{\eta'} \omega_0 =V$, i.e. the pair $(f,\eta')$ is a solution of equations \eqref{eq:Yangreductiond=2BIS} with the parameter $\varepsilon^2=1$ and $\int_{S^2}e^{\eta'}\omega_0=V$. Therefore, by Proposition \ref{prop:Futaki} and Lemma \ref{lem:EMHFutakiev} we necessarily have
$$
\cF_{\ell,N,V,\tau} = ia \left(  V- \frac{4\pi N}{\tau^2} \right) (N- 2\ell) = 0.
$$
Finally, by Lemma \ref{lem:BNGP}, we have $N = 2 \ell$, and therefore we conclude that
$$
d = 2 \qquad \textrm{and} \qquad n_1 = n_2 = \ell = \frac{N}{2}. 
$$
\end{proof}

\begin{remark}\label{rem:HLS}
Our main result in Theorem \ref{thm:dleq2} gives infinitely many examples of configurations of points on the sphere for which there cannot be solutions to the self-dual Einstein-Maxwell-Higgs equations, in apparent contradiction with \cite[Theorem 1.2]{HLS} (see the Erratum \cite{HLSErratum}). For example, one can take $d = 1$ and arbitrary $N$, or $d = 2$ and $n_1 > n_2$.

\end{remark}

\end{document}